\documentclass[letter,10pt]{article}
\usepackage{cite}
\usepackage{eucal}
\usepackage{kotex}
\usepackage{lineno,hyperref}
\usepackage{epstopdf}
\usepackage{pifont}
\usepackage{latexsym}
\usepackage{graphicx}
\usepackage{float}
\usepackage[dvips]{xy}
\usepackage{ifpdf}
\usepackage{multirow}
\usepackage{amssymb, amsthm, amsmath}
\usepackage{hhline}
\usepackage{centernot}
\usepackage{verbatim, comment, color}
\usepackage{enumerate}
\usepackage{setspace}
\usepackage{lipsum}
\usepackage{subcaption}
\usepackage{bbm}
\usepackage{comment}

\usepackage{mathtools}

\DeclareOldFontCommand{\bf}{\normalfont\bfseries}{\mathbf}

\newtheorem{theorem}{Theorem} 
\newtheorem{proposition}[theorem]{Proposition}
\newtheorem{lemma}[theorem]{Lemma}
\newtheorem{remark}[theorem]{Remark}

\newtheorem{definition}[theorem]{Definition}

\DeclareMathOperator*{\argmin}{arg\,min}

\newcommand{\be}{\begin{equation}}
\newcommand{\ee}{\end{equation}}
\newcommand{\bn}{\begin{enumerate}}
\newcommand{\en}{\end{enumerate}}
\newcommand{\bi}{\begin{itemize}}
\newcommand{\ei}{\end{itemize}}

\newcommand{\R}{{\mathbb R}}

\newcommand{\cP}{{\mathcal P}}

\newcommand{\al}{\alpha}
\newcommand{\bt}{\beta}

\newcommand{\Ga}{\Gamma}
\newcommand{\de}{\delta}

\newcommand{\ep}{\varepsilon}
\newcommand{\ka}{\kappa}
\newcommand{\la}{\lambda}

\newcommand{\vp}{\varphi}

\newcommand{\ot}{\otimes}
\newcommand{\co}{\preceq_c}

\newcommand{\tsp}{\text{supp}}

\newcommand{\PiM}{\Pi_M}
\newcommand{\Proj}{\text{Proj}}

\newcommand{\Pibar}[1]{\Pi^{\text{Bar}}(#1)}

\definecolor{darkspringgreen}{rgb}{0.09, 0.45, 0.27} 
\definecolor{darkgray}{rgb}{0.66, 0.66, 0.66}

\numberwithin{equation}{section}
\numberwithin{theorem}{section}

\title{Remarks on multi-period martingale optimal transport\thanks{The work of JH is completed in partial fulfillment of the requirements for a doctoral degree at the University of Alberta. BP is pleased
		to acknowledge support from Natural Sciences and Engineering Research
		Council of Canada Grants 04658-2018 and and 04864-2024.}}
\author{Brendan Pass\thanks{
		Department of Mathematical and Statistical Sciences, University of Alberta,
		Edmonton, Alberta, Canada pass@ualberta.ca.}, Joshua Hiew\thanks{
		Department of Mathematical and Statistical Sciences, University of Alberta,
		Edmonton, Alberta, Canada pass@ualberta.ca.} }
\date{}

\begin{document}

\maketitle

\begin{abstract}
	We study the structural properties of multi-period martingale optimal transport (MOT).  
We develop new tools to address these problems, and use them to prove several uniqueness and structural results on three-period martingale optimal transport.  More precisely, we establish lemmas on how and when two-period martingale couplings may be glued together to obtain multi-period martingales and which among these glueings are optimal for particular MOT problems.  We use these optimality results to study limits of solutions under convergence of the cost function and obtain a corresponding linearization of the optimal cost.  We go on to establish a complete characterization of limiting solutions in a three-period problem as the interaction between two of the variables vanishes.  Under additional assumptions, we show uniqueness  of the solution and a  structural result which yields the solution essentially explicitly.  For the full three-period problem, we also obtain several structural and uniqueness results under a variety of different  assumptions on the marginals and cost function.

We illustrate our results with a real world application, providing approximate model independent upper and lower bounds for options depending on Amazon stock prices at three different times.  We compare these bounds to prices computed using certain models.
%
    
\end{abstract}

\section{Introduction}

\subsection{Background and Motivation}

Martingale optimal transport (MOT) is an optimization problem with important applications in operations research and financial engineering \cite{BeiglbockHenryPenkner2013, HenryLabordere, NutzStebeggTan20}.  Mathematically, it extends the classical optimal transport problem \cite{santambrogio2015, Villani09, Galichon2016} by adding an additional martingale constraint to the coupling.

The motivation arises from financial engineers' desire to derive model independent bounds for prices of derivatives which are consistent with observed market data.  Consider a derivative whose payoff depends on the price of an asset at several different future times.  The future values of the asset are of course not known, but risk neutral distributions of the prices can be reconstructed from traded prices of vanilla options on that asset at each future time \cite{BreedenLitzenberger}.  These options are heavily traded, and so prices for many of them are typically known; this data can be used to estimate the single time distribution of the asset price. 

On the other hand, the price of the derivative whose payoff depends on prices at several times depends on the dependence structure, or coupling, of these single time distributions, and this cannot typically be determined from available data. The price therefore cannot be pinned down uniquely from market data.  The model free pricing problem is to find the \emph{minimum} (and maximum, via a similar problem, though we focus on the minimum here) possible prices among all multi-period margingale (to conform with the no arbitrage condition of derivative pricing) distributions which have the known single time distributions as its marginals (to be consistent with the data); a more detailed discussion can be found in \cite{HenryLabordere}. The precise mathematical statement of the problem, known as martingale optimal transport, requires a bit of notation and is formulated below \eqref{prob:Primal-multi-period-MOT}.  As a linear program, this problem also has a dual, \eqref{prob:Dual-multi-period-MOT}, which has a complementary financial interpretation in terms of hedging strategies (see the discussion in the following subsection). 

A natural goal is to understand the structure of solutions,  allowing practitioners to quickly and accurately compute\footnote{Ideally this can be done in closed form; more realistically, it can be done numerically exploiting structural features of the solution.} the desired bounds from the available data.  In the simplest case, when only two time periods are involved ($n = 2$ in \eqref{prob:Primal-multi-period-MOT} below), this question has been studied extensively, and a fairly complete understanding of solutions has emerged \cite{bj}.  Real world derivatives, however, often depend on several time periods ($n \geq 2$ in \eqref{prob:Primal-multi-period-MOT}).  This multi-period MOT problem is very delicate.  When the payoff (or cost) function decouples in a particular way, more precisely, when $c(x_1,x_2,...,x_n) =\sum_{i=2}^nc_i(x_1,x_i)$ in \eqref{prob:Primal-multi-period-MOT} and the $c_i$ satisfy certain assumptions, the structure is completely understood \cite{NutzStebeggTan20}. However, these costs are very special, as each variable interacts only with $x_1$. For more general costs, in particular, those for which all pairs of variables interact,  as is realistic in many applications, to the best of our knowledge, nothing is known.  The purpose of this paper is to shed some light on the structure of solutions to these challenging problems, albeit under various simplifying assumptions, and to apply the resulting insights to pricing problems using real-world data.

\subsection{Problem Formulation}
Let $\cP(X)$ denote the set of probability measures on a space $X \subset \R^d$. 
For each $i = 1, \dots, n$, let $\mu_i \in \cP(X_i)$ be a probability measure supported on a compact set $X_i \subset \R^d$,
and define the space $X \coloneqq X_1 \times \dots \times X_{n}$.\footnote{The motivating pricing problem described above corresponds to working in one dimension, $d=1$.  We formulate the problem for a general $d$ here, but will specialized to $d=1$ in some later sections.} 
Assume that $\mu_1, \dots, \mu_{n}$ satisfy the convex order condition, which we will denote by $\mu_i \co \mu_{i+1}$, defined by
\be \label{eqn: convex order}
\int \vp(x) d\mu_i(x) \leq \int \vp(x) d\mu_{i+1}(x), \quad \forall \text{ convex } \varphi : \mathbb{R}^d \to \mathbb{R}.
\ee

The multi-period MOT problem seeks to minimize an intertemporal cost function subject to martingale constraints. 
Let $c: X \to \R$ be a continuous cost function.  For a probability measure $\pi \in \cP(X)$, we denote by $\Proj_{I}(\pi)$ the projection of  $\pi$ onto the coordinates indexed by $I \subset \{1, \dots, n\}$.  The set of couplings is then defined by $\Pi(\mu_1,\mu_2,...,\mu_n): =\{\pi \in \cP(X): \Proj_i(\pi) =\mu_i\}$.  For a given $\pi \in \cP(X)$, we will often consider the disintegration with respect to certain sets of variables $I\subset \{1, \dots, n\}$; setting $\mu_I = \Proj_I(\pi)$ we will write
$$
\pi = \mu_I \otimes \kappa_I^{I^C}
$$
where $\ka_I^{I^C}(\{x_i\}_{i\in I},\{dx_i\}_{i\in I^C})$ is the conditional probability of the variables $\{x_i: i \in I^C\}$ indexed by the complement $I^C$ of $I$, given the variables $x_i$ for $i \in I$.  

The set of martingale couplings of the $\mu_i$ is then defined by

\begin{eqnarray*}
\PiM (\mu_1, \dots, \mu_{n})
:=&\{\pi \in \Pi(\mu_1,\mu_2,...,\mu_n): \int_{X_{i+1} \times ...\times X_n}x_{i+1}\ka_{1,2,...i}^{i+1,...,n}(x_1,x_2,...x_i,dx_{i+1},...,dx_n) =x_i\\
&\forall i =1,2,...n-1, \text{ where }\pi  =\mu_{12,...,i} \otimes \ka_{1,2,...i}^{i+1,...,n}\}
\end{eqnarray*}

By Strassen’s theorem \cite{strassen65}, the convex order condition \eqref{eqn: convex order} ensures that $\PiM(\mu_1, \dots, \mu_{n})$ is non-empty.
The primal MOT problem is given by
\be\label{prob:Primal-multi-period-MOT}
P(\mu_1, \dots, \mu_{n}) = \inf_{\pi \in \PiM (\mu_1, \dots, \mu_{n})} \int c(x_1, \dots, x_{n}) d\pi.
\ee

The dual formulation of multi-period MOT plays a crucial role in financial applications, 
as it provides a natural interpretation in terms of hedging strategies. 
The dual problem is given by
\be\label{prob:Dual-multi-period-MOT}
D(\mu_1, \dots, \mu_{n}) = \sup_{(u_i), (h_i)} \sum_{i=1}^{n} \int u_i(x_i) d\mu_i(x_i),
\ee
where the supremum is taken over functions $(u_i: X_i \to \R \cup \{+\infty\})$ and $(h_i : X_1 \times ...X_i \to \R)$ satisfying 
\[
\sum_{i=1}^{n} u_i(x_i) + \sum_{i=1}^{n-1} h_i(x_1, \dots, x_i) (x_{i+1} - x_i) \leq c(x_1, \dots, x_n).
\]

On the dual side, the MOT problem can be understood as constructing an \textit{executable semi-static trading strategy} that sub-replicates the contingent claim cost function $c$ \cite{HenryLabordere}. 
The functions $u_i(x_i)$ represent the \textit{payoffs} of European options written on the asset prices at each time step,
while $h_i(x_1, \dots, x_i)$ correspond to \textit{predictable processes} that determine dynamic trading strategies. 
The dual value of the problem thus represents the \textit{robust sub-replication price} of the contingent claim.  

\subsection{Our Contributions}
We begin by developing some basic tools to study multi-period martingale optimal transport, including glueing lemmas characterizing when and how certain two period martingales can be glued to obtain multi-period ones (see Lemmas \ref{lem:GluingLemma1} and \ref{lem:GluingLemma2} below), and, as a consequence, some basic results on solutions to problems with certain decoupled cost functions (Propositions \ref{prop:ConsecutiveOptimality} and \ref{prop:SharedInitialOptimality}).  We also establish some preliminary results on limits of optimal plans as cost functions converge in a certain way (Proposition \ref{prop:Convergence}), and derivatives of the total cost under corresponding perturbations (Proposition \ref{prop: derivatives of optimal cost}); in particular, for certain problems, this can be used to find a linear approximation of the model-free price bound around points where it can be computed in essentially closed form (see Remark \ref{rem: linearization} below).

We go on to apply these tools to several three period ($n=3$) problems.  First, we provide a complete characterization of the limit of solutions as the interaction between the first and third time period vanishes, in terms of a novel variant of the martingale optimal transport problem between conditional probabilities (a financial interpretation of this problem is offered as well); see Theorem \ref{thm: optimizers of limiting three period problem}.  For costs with a particular structure, we obtain a further structural result on solutions (Theorem \ref{thm: structure of optimizer in three period limit problem}), which allows for the construction of explicit solutions when appropriate two marginal problems can be solved in closed form (as is the case for a wide variety of cost functions \cite{bj} \cite{HenryLabordereTouzi19}), as well as establish their uniqueness (Theorem \ref{thm: left monotone uniqueness}).  

We then turn to the full three  marginal problem and establish several uniqueness and structural results, under various assumptions on the marginals; see Theorems \ref{discMarUniq}, \ref{thm: uniqueness for triply supported y} and \ref{thm: uniqueness for discrete x and y}.  Though models satisfying the assumptions required in these results are admittedly highly idealized, to the best of our knowledge they represent the first uniqueness and structural results for multi-period MOT problem with cost involving interactions between all pairs of variables.  We hope and expect that they will initiate a line of research leading to more refined results on these problems in the future.

We also develop an application of our theoretical results to real world pricing problems, providing an approximation of the robust price bounds of path-dependent derivatives of the form of sums of pairwise payoffs.  We apply this method to real world data on Amazon stock prices,  finding approximate bounds on the risk-neutral third moment (and hence the skewness) of the sum of prices at different times as well as a basket of straddles.  We compare these results to prices computed using particular modeling assumptions, and verify that the model prices fall within the approximate model independent bounds.




\subsection{Structure of the Paper}
The remainder of this paper is organized as follows. 
In Section 2, we establish some preliminary results for multi-period MOT which we will use later on,
including martingale gluing lemmas and a cost perturbation analysis. 
Section 3 applies these tools to characterize the limiting solution to a three period problem as the interaction cost between the first and last variable vanishes.
In Section 4, we present new structural results for three period MOT, 
including uniqueness theorems for optimal couplings under different assumptions on the cost function and marginals. 
Applications to derivative pricing problems using real world data are presented in Section 5.

\section{Preliminary definitions and results}
This section develops certain preliminary results we will need later on.  

\subsection{Gluing lemmas and optimality consequences}\label{sec:GluingLemmas}

This section presents two \textit{martingale gluing lemmas}, 
which establish conditions under which a sequence of two-period couplings can be combined into a valid multi-period martingale transport plan.


The following variant of the martingale condition will arise naturally below.  Given a mapping $F: X_i \rightarrow \mathbb{R}^d$, we define 

\[
\Pibar{F,\mu_i,\mu_j} := \Big\{ \pi =\mu_i\otimes \ka^j_i \in \Pi(\mu_i,\mu_j) \, : \, 
\int x_j \, d \ka^j_i (x_j)= F(x_i)\quad \text{ for }\mu_i \text{ a.e. }x_i \in  X_i\Big\}.
\]
Note that  if $F$ is the identity mapping, $F(x) =x$, $\Pibar{F,\mu_i,\mu_j} =\PiM(\mu_i,\mu_j)$.  The case where $F$ is a constant mapping will also play an important role in what follows.


\begin{lemma}\label{lem:GluingLemma1}
    \textbf{(Martingale gluing lemma I)}  
    Let $\mu_1 \co \mu_2 \co \mu_3$ be probability measures in convex order. 
	Suppose that $\pi^{12}=\mu_2\otimes \kappa ^1_2 \in \Pi_M(\mu_1, \mu_2)$ and $\pi^{23}=\mu_2\otimes \kappa^3_2  \in \Pi_M(\mu_2, \mu_3)$ are two martingale couplings. 

    Then the set of martingale couplings $\pi^{123} \in \Pi_M(\mu_1, \mu_2, \mu_3)$  such that
    \[
    \Proj_{12}(\pi^{123}) = \pi^{12}, \quad \Proj_{23}(\pi^{123}) = \pi^{23}.
    \]
    is given by
    $$
    \{\pi^{123} \in \Pi_M(\mu_1, \mu_2, \mu_3): \pi^{123}=\mu_2\otimes \ka_2^{13}, \ka_2^{13}(x_2,dx_1,dx_2) \in \Pibar{F_{x_2}, \ka_2^1,\ka_2^3} \text{ for } \mu_2 \text{ a.e. } x_2\}
    $$
    where, for each fixed $x_2$, $F_{x_2}:X_1 \rightarrow X_3$ is the constant function, $F_{x_2}(x_1) =x_2$.
\end{lemma}

\begin{proof}
Since disintegrating $\pi^{123}=\mu_2\otimes \ka_2^{13}$ with respect to $x_1$ and $x_2$ is equivalent to disintegrating $\ka_2^{13} = \ka_2^1 \otimes \ka_{21}^{3}$ with respect to $x_1$, we see that the martingale conditions is equivalent to $\int x_3\ka_{21}^{3}(x_1,x_2,dx_3)=x_2$, which is exactly the condition characterizing $\Pibar{F_{x_2}, \ka_2^1,\ka_2^3}$.
\end{proof}
\begin{remark}\label{rem: markovian glueing}
The set of such couplings is always non-empty, since we can take $ \kappa^{13}_2 = \kappa^{1}_2 \kappa^{3}_2  $ to be product measure.  This is in fact the only glueing which is also Markovian.
\end{remark}

Using Lemma \ref{lem:GluingLemma1} successively, one gets a characterization of the ways to glue $n-1$ pairs of martingale $2$ period couplings $\pi^{i,i+1} \in \PiM(\mu_i, \mu_{i+1})$, $i=1,2..,n-1$ to obtain an $n$-period martingale $\pi^{12..n} \in \PiM(\mu_1,... \mu_{n})$. The following result asserts that, when the $\pi^{i,i+1}$ are all optimal for two-period problems,  each such martingale is optimal in the $n$-period MOT problem for an appropriate cost function.


\begin{proposition}
    \label{prop:ConsecutiveOptimality}
    Let $\mu_1 \co \dots \co \mu_{n}$ be probability measures in convex order. 
	For each $i = 1, \dots, n-1$, 
	let be an optimal martingale coupling $\pi^*_{i, i+1}\in \Pi_M(\mu_i, \mu_{i+1})$ for the two-period MOT problem with continuous cost function $c_i(x_i, x_{i+1})$. 
    Then, any coupling $\pi^* \in \Pi_M(\mu_1, \dots, \mu_{n})$ constructed by successively applying Lemma \ref{lem:GluingLemma1} is optimal for the multi-period MOT problem with cost
    \[
    c(x_1, \dots, x_{n}) = \sum_{i=1}^{n-1} c_i(x_i, x_{i+1}).
    \]

Conversely, if $\pi^{12,..n} \in \PiM(\mu_1,... \mu_{n})$ is optimal for the multi-period MOT problem with this cost, each twofold projection $Proj_{i,i+1}(\pi^{12,..n})$ is optimal for the corresponding $2$-period MOT problem.
    
\end{proposition}

\begin{proof}


    The result follows easily by noting that for any $\pi \in \PiM(\mu_1,...,\mu_n)$, we have
    $$
\int    \sum_{i=1}^{n} c_i(x_i, x_{i+1})d\pi =\sum_{i=1}^{n} \int c_i(x_i, x_{i+1})d(\Proj_{i,i+1}(\pi))
    $$
\end{proof}



We now turn to a second gluing problem, 
where instead of working with adjacent marginals $(\mu_1, \mu_2, \mu_3)$, 
we consider two couplings $\pi^{12} \in \Pi_M(\mu_1, \mu_2)$ and $\pi^{13} \in \Pi_M(\mu_1, \mu_3)$. 





\begin{lemma}\label{lem:GluingLemma2}
    \textbf{(Martingale gluing lemma II)}  
    Let $\mu_1 \co \mu_2 \co \mu_3$ be probability measures in convex order. 
	Suppose that $\pi^{12} = \mu_1 \otimes \kappa_1^2 \in \Pi_M(\mu_1, \mu_2)$ and $\pi^{13} = \mu_1 \otimes \kappa_1^3 \in  \Pi_M(\mu_1, \mu_3)$ are two martingale couplings such that for $\mu_1$ almost every $x_1$ we have $\kappa_1^2(x_1,dx_2) \co \ka_1^3(x_1,dx_3)$
    Then, there exists a martingale coupling $\pi^{123} \in \Pi_M(\mu_1, \mu_2, \mu_3)$ such that
    \[
    \Proj_{12}(\pi^{123}) = \pi^{12}, \quad \Proj_{13}(\pi^{123}) = \pi^{13}.
    \]
\end{lemma}

\begin{proof}
    By Theorem 1.3 in \cite{LeskelaVihola17}, 
	there exists a kernel $\kappa_1^{23}(x_1, dx_2, dx_3)$ that is a martingale coupling between $\kappa_1^2(x_1, dx_2)$ and $\kappa_1^3(x_1, dx_3)$ for $\mu_1$ a.e. $x_1$ (alternatively, we may obtain this by applying Strassen's theorem pointwise). 

    We then construct the joint law $\pi^{123}$ as:
    \[
    \pi^{123}(dx_1,dx_2,dx_3) = \mu_1(dx_1) \otimes \kappa_1^{23}(x_1, dx_2, dx_3).
    \]
    By construction, $\pi^{123}$ is a martingale coupling between $\mu_1, \mu_2, \mu_3$ and satisfies the projection constraints.
\end{proof}


  As above, given $n-1$ martingale couplings $\pi^{1i} =\mu_1\otimes \ka^i_1$ with conditional probabilities in convex order $\ka^i_1 \co \ka^{i+1}_1$ $\mu_1$ a.e., we can successively apply Lemma \ref{lem:GluingLemma2} to construct a martingale coupling $\pi^{12...n} \in \PiM(\mu_1,...,\mu_n)$ such that $Proj_{1i}(\pi^{12...n} ) =\pi^{1i}$.  The result below asserts optimality of these couplings for certain MOT problems.


\begin{proposition}\label{prop:SharedInitialOptimality}
    Let $\mu_1 \co \dots \co \mu_{n+1}$ be probability measures in convex order. 
	Suppose that 
	there exists martingale coupling $\pi^*_{1,i} =\mu_1 \otimes \ka_1^i \in \Pi_M(\mu_1, \mu_i)$ for each $i = 2, \dots, n+1$ which are  optimal for the two-period problems with continuous costs $c_i(x_1, x_i)$ such that $\ka^i_1 \co \ka^{i+1}_1$ $\mu_1$ a.e.
	
    Then, any coupling $\pi^* \in \Pi_M(\mu_1, \dots, \mu_{n+1})$ constructed using succesive iterations of Lemma \ref{lem:GluingLemma2} is optimal for the multi-period MOT problem with cost
    \be\label{eqn: shared initial cost}
    c(x_1, \dots, x_{n}) = \sum_{i=2}^{n} c_i(x_1, x_i).
    \ee

    Conversely, if $\pi^*$ is optimal for the multi-period MOT problem, each of its projections $Proj_{1i}(\pi^{12...n} ) $ is optimal for the corresponding two period problem.
\end{proposition}

\begin{proof}
The proof is similar to the proof of Proposition \ref{prop:ConsecutiveOptimality}; it follows immediately after noting that for any $\pi \in \PiM(\mu_1,...,\mu_n)$, we have
    $$
\int    \sum_{i=1}^{n} c_i(x_1, x_{i})d\pi =\sum_{i=1}^{n} \int c_i(x_1, x_{i})d(\Proj_{1,i}(\pi))
    $$


\end{proof}
\begin{remark}
    When $d=1$, under additional assumptions on the cost functions $c_i$, optimizers for the MOT problem with cost \eqref{eqn: shared initial cost} are completely characterized in \cite{NutzStebeggTan20}.  The preceding proposition provides only a partial characterization, since the construction  requires the optimal twofold marginals $\pi_{1,i}^*$, and requires the strong conditional convex order condition.  However, it also applies in higher dimensions $d \geq 1$ and does not require any structural assumptions on the $c_i$.
\end{remark}
\subsection{Limiting behaviour for converging cost  functions}\label{sec:Convergence}

In this subsection, we study the limit of optimal martingale couplings in the limit as a perturbation of the cost function vanishes.
Suppose that $c, p \in C(X)$ are continuous cost functions, 
and consider the perturbed cost family:
\[
c_\ep(x) = c(x) + \ep p(x).
\]

Let $\Pi_{\ep}^M(\mu_1,...,\mu_n) :=\argmin_{\pi \in \Pi^M(\mu_1,...,\mu_n)} \int c_\ep \, d\pi$ be the set of optimal measures for the cost $c_{\ep}$
and consider a  sequence $\ep_k \to 0$ with $\ep_k > 0$.

\begin{proposition}\label{prop:Convergence}
    For each $k$, let $\pi_k \in \Pi_{\ep_k}^M(\mu_1,...,\mu_n)$
    Any weak limit (after relabeling if necessary) $\pi_0 = \lim \pi_k$ belongs to $\Pi_{0}^M(\mu_1,...,\mu_n)$ and minimizes the cost function:
    \[
    \inf_{\pi \in \Pi_{0}^M(\mu_1,...,\mu_n)} \int p \, d\pi.
    \]
\end{proposition}

\begin{remark}
    Since the set of martingale couplings is compact in the weak topology \cite[Proposition 4.4]{BeiglbockHenryPenkner2013}, 
	every sequence $\{\pi_k\}$ has a convergent subsequence.
\end{remark} 

\begin{proof}
By optimality of $\pi_k$, for any $\pi \in \Pi^M(\mu_1,...,\mu_n)$,
\[
\int (c + \ep_k p) \, d\pi_k \leq \int (c + \ep_k p) \, d\pi.
\]
Taking the limit as $k \to \infty$ and using the boundedness and continuity of $c$ and $p$, we obtain
\[
\int c \, d\pi_0 \leq \int c \, d\pi.
\]
Thus, $\pi_0 \in \Pi_{0}^M(\mu_1,...,\mu_n)$.

For any other $\pi \in \Pi_{0}^M(\mu_1,...,\mu_n)$, we again use optimality of $\pi_k$:
\be\label{eq:optimality}
    \int (c + \ep_k  p) \, d\pi_k 
    \leq \int (c + \ep_k p) \, d\pi 
\ee
Since $\pi \in \Pi_{0}^M(\mu_1,...,\mu_n)$, we have $\int c \, d\pi_k \geq \int c \, d\pi$.
Combining this with (\ref{eq:optimality}) gives $\int\ep_k p \, d\pi_k \leq \int \ep_k p \, d\pi$.
Since $\ep_k  > 0$, we divide by $\ep_k$ and take limits to conclude
\[
\int p \, d\pi_0 \leq \int p \, d\pi.
\]
Since $\pi \in \Pi_{0}^M(\mu_1,...,\mu_n)$ was arbitrary, the result follows.
\end{proof}


\begin{remark}
This result is in contrast with the instability of MOT with respect to perturbations of the \emph{marginals} demonstrated in \cite{BruckerhoffJuillet22} (for $d>1$).
\end{remark}

We now analyze how the perturbed optimal value function behaves under small perturbations. Define
\[
P(\ep) = \inf_{\pi \in \Pi^M(\mu_1,...,\mu_n)} \int (c + \ep p) \, d\pi.
\]
Let $P'_+(0)$ and $P'_-(0)$ denote the right and left derivatives at $\ep = 0$, respectively.

\begin{proposition}\label{prop: derivatives of optimal cost}
    The function $P(\ep)$ is right (left)-differentiable at $0$, with 
    $$
     P'_+(0) = \inf_{\pi \in \Pi_{0}^M(\mu_1,...,\mu_n)} \int p \, d\pi, \text{ } P'_-(0) = \sup_{\pi \in \Pi_{0}^M(\mu_1,...,\mu_n)} \int p \, d\pi,
    $$
    In particular, if there exists a unique optimal $\pi_0 \in \Pi^M_0(\mu_1,...,\mu_n)$ for $c$, then $P(\ep)$ is differentiable at $0$ and
    \be\label{eq:primal_derivative}
    P'(0) = \int p \, d\pi_0.
    \ee
\end{proposition}

\begin{proof}
Since $P(\ep)$ is the infimum of affine functionals, 
it is concave in $\ep$. Thus, it is differentiable almost everywhere and has well-defined one-sided derivatives.

At differentiable points, the Envelope Theorem implies
\[
P'(\ep) = \int p \, d\pi_\ep,
\]
for each optimizer $\pi_\ep \in \Pi_\ep^M(\mu_1,...,\mu_n)$.

Let $\{\ep_k\}$ be a sequence with $\ep_k \to 0^+$ where $P(\ep)$ is differentiable at each $\ep_k$. 
Denote the corresponding optimizer by $\pi_k$. Proposition \ref{prop:Convergence} implies the desired formula for $P'_+(0)$.  A very similar argument yields the formula for $P'_-(0)$.  

If the optimizer for $\ep =0$ is unique, that is, if $\Pi_{0}^M(\mu_1,...,\mu_n)$ is a singleton, then the left and right hand derivatives are equal, in which case $P$ must be differentiable at $0$.



\end{proof}

\section{Limiting three marginal problems and a transport problem for conditional probabilities}

In the remainder of the paper, we restrict our attention to the three-period MOT problem 
in order to shed some light on the structure of solutions to multi-period MOT problems.
To streamline notation, we use $\mu_X \in \cP(X), \mu_Y \in \cP(Y)$, and $\mu_Z \in \cP(Z)$ to denote the marginals,
where $X, Y, Z \subset \R$, and
$(x, y, z)$ in place of $(x_1, x_2, x_3)$ for the state variables. 

We will mostly focus on cost functions of the form 
\begin{equation}\label{eqn: 3 period pairwise cost}
c(x, y, z) = c_1(x, y) + c_2(y, z) + c_3(x, z).
\end{equation}
In this section, we consider perturbations around $c_3=0$.  The following section allows for more general $c_3$, but restricts to marginals of very particular forms.
\subsection{Localized problem for conditional probabilities}
Consider the perturbed cost function $c_\ep(x, y, z) = c_1(x, y) + c_2(y, z) + \ep c_3(x, z)$, for continuous $c_1,c_2,c_3$ with $\epsilon >0$.

We begin by introducing a variant of the MOT problem.  Given two measures $\sigma_X \in \cP(X)$ and $\sigma_Z \in \cP(Z)$, set $\bar z =\int zd\sigma_Z(z)$, this problem is to minimize 
\begin{equation}\label{eqn: fixed barycenter problem}
\min_{\pi \in \Pibar{F,\sigma_X,\sigma_Z}}\int c_3(x,z)d\pi(x,z)
\end{equation}
where $F(x) =\bar z$ is the constant function.  

\begin{theorem}\label{thm: optimizers of limiting three period problem}
    Let $\pi_0 =\mu_Y \otimes \ka^{XZ}_Y$ be a limit point of solutions $\pi_\ep$ to the  $3$ period MOT problem with cost $c_\epsilon$.  Then for $\mu_Y$ almost every $y$ the conditional probabilities $\ka^{XZ}_Y$ are optimal in \eqref{eqn: fixed barycenter problem} for marginals $\sigma_X=\ka^X_Y(y,dx)$ and $\sigma_Z=\ka_Y^Z(y,dz)$ and constant function $F(x)=y$ where $\ka^X_Y(y,dx)$ and $\ka_Y^Z(y,dz)$ are conditional probabilities of optimal measures $\pi^{XY} =\mu_Y \otimes\ka^X_Y $ and $\pi^{YZ} =\mu_Y \otimes\ka^Z_Y $  in the $2$ period MOT problems between $\mu_X$ and $\mu_Y$ with cost $c_1$ and $\mu_Y$ and $\mu_Z$ with cost $c_2$, respectively.
\end{theorem}
\begin{proof}
    Propositions \ref{prop:ConsecutiveOptimality} and \ref{prop:Convergence} imply that $\pi^{XY}:=\Proj_{XY}(\pi_0)$ and  $\pi^{YZ}:=\Proj_{YZ}(\pi_0)$ are optimal in the corresponding $2$ period MOT problems.  Furthermore, among all other martingale measures $\tilde \pi=\mu_Y \otimes \tilde \ka^{XZ}_Y \in \PiM(\mu_X,\mu_Y,\mu_Z)$ sharing the same overlapping marginals, $\Proj_{XY}(\tilde \pi) = \pi^{XY}$, $Proj_{YZ}(\tilde \pi) = \pi^{YZ}$, $\pi_0$ minimizes
    \begin{equation}\label{eqn: pointwise barycenter problem}
    \int c_3(x,z) d\tilde\pi =\int_Y \big(\int_{X \times Z} c_3(x,z) \tilde\ka^{XZ}_Y(y,dxdz)\big)d\mu_Y(y)
    \end{equation}
    The constraints $\Proj_{XY}(\tilde \pi) = \pi^{XY}$ and $\Proj_{YZ}(\tilde \pi) = \pi^{YZ}$ correspond to $\Proj_{X}(\tilde \ka^{XZ}_Y(y,dxdz)) = \ka^{X}_Y(y,dx)$ and $\Proj_{Z}(\tilde \ka^{XZ}_Y(y,dxdz)) = \ka^{Z}_Y(y,dz)$, respectively, for $\mu_Y$ almost every $y$, while the constraint that $\tilde \pi$ is a $3$ period martingale then corresponds to $\tilde \ka^{XZ}_Y \in \Pibar{F,\ka^{X}_Y, \ka\mu_Y}$.

    Therefore, minimizing the left hand side of \eqref{eqn: pointwise barycenter problem} is equivalent to minimizing the integrand $\int_{X \times Z} c_3(x,z) \tilde\ka^{XZ}_Y(y,dxdz)$ among $\tilde\ka^{XZ}_Y \in \Pibar{F,\ka^{X}_Y, \ka^Z_Y}$ for $\mu_Y$ almost every $y$, as desired.
\end{proof}
\begin{remark}\label{rem: interpretation of overlapping marginals MOT}

Aside from providing approximations of optimizers for the three period MOT problem \eqref{prob:Primal-multi-period-MOT} with cost \eqref{eqn: 3 period pairwise cost} when $c_3$ is small compared to $c_1$ and $c_2$, the MOT problem with overlapping marginals, captured by the minimization \eqref{eqn: pointwise barycenter problem} among measures $\tilde \pi \in \PiM(\mu_X,\mu_Y,\mu_Z)$ with $\Proj_{XY}(\tilde \pi) = \pi^{XY}$, $\Proj_{YZ}(\tilde \pi) = \pi^{YZ}$ has another natural financial interpretation.  

 In certain situations, the couplings $\pi^{XY}$ between the first and second time, and $\pi^{YZ}$ between the second and third times, are known, or can at least be estimated from market data. This situation occurs, for example, when there are enough rainbow options to estimate the joint distributions at consecutive maturities (1,2) and (2,3) \cite{TalponenViitasaari2014}, 
but sufficient such data  on (1,3) is lacking. 
In such situations, the problem above arises as the model independent pricing problem for a derivative with payoff $c_3(x,z)$ depending on values at the first and third time. 

A dual problem and corresponding duality result can easily be deduced, using Theorem 2.1 in \cite{zaev2015}.  The dual formulation corresponds to constructing a semi-static portfolio that subreplicates the cost function, ensuring that the optimal value is achieved through an implementable trading strategy.

\end{remark}
\begin{remark}\label{rem: linearization}
     Combined with Proposition \ref{prop: derivatives of optimal cost}, Theorem \ref{thm: optimizers of limiting three period problem} in fact yields a \emph{linearization} of the optimal cost (or, in terms of the model-free pricing application, the bound on the derivative price) for the cost $c_\ep(x, y, z) = c_1(x, y) + c_2(y, z) + \ep c_3(x, z)$ near $\epsilon =0$, in terms of problems which can often be solved explicitly.  We exploit this point of view to obtain approximations of model independent price bounds for real world data in Section 5.
\end{remark}

\subsection{Structure of optimal pointwise couplings}\label{sect: pointwise coupling structure}
This subsection examines the structure of optimal solutions to problem \eqref{eqn: fixed barycenter problem}; we use the notation $c(x,z)$ in place of $c_3(x,z)$ here to address \eqref{eqn: fixed barycenter problem} in isolation.  More precisely, for one dimensional marginals, $d=1$, we establish a characterization that allows us to solve this problem explicitly, and, as a consequence of Theorem \ref{thm: structure of optimizer in three period limit problem}, construct solutions to the limiting three period problem whenever the optimal two period measures $\pi^{XY}$ and $\pi^{YZ}$ are known, as is the case for a reasonably wide class of two period costs (see, for example, \cite{bj, HenryLabordereTouzi19}).

\begin{definition}
	A set $\Ga \subset \R^2$ is left-monotone if it satisfies the no-crossing condition: 
	for any $(x, z^-), (x, z^+), (x', z') \in \Ga$ with $z^- < z^+$ and $x < x'$, it holds that $z' \notin (z^-, z^+)$.
	A coupling $\pi \in \Pi(\sigma_X, \sigma_Z)$ is left-monotone if its support $\Ga$ is left-monotone. 
\end{definition}

This structure is well-known in classical MOT \cite{bj}, \cite{HenryLabordereTouzi19} when the cost function is of \emph{martingale Spence–Mirrlees} type, 
i.e., $\partial_x c(x, z)$ is strictly concave in $z$ for each $x$, or more generally, $c(x', z) - c(x, z)$ is strictly concave in $z$ for each $x < x'$.
We establish that optimal couplings in the pointwise problem \eqref{eqn: fixed barycenter problem} are left-monotone under the martingale Spence–Mirrlees condition on $c$.

To do this, we use $(c, W)$-monotonicity, 
a generalization of cyclical monotonicity \cite{zaev2015}.  Let $W = \{ h(x)(z - x) : h \in C(X) \}$. We define an equivalence relation $\sim_W$ on $\mathcal{P}(\mathbb{R}^2)$ by saying that two measures $\alpha$ and $\beta$ are competitors, 
denoted $\alpha \sim_W \beta$, if they have the same marginals and satisfy $\int f \, d\alpha = \int f \, d\beta$ for all $f \in W$.  A set $\Gamma \subset \R^2$ is called \emph{$(c, W)$-monotone} if for any finite collection of points $S = \{(x_i, z_i) \subset \Gamma\}$ 
and any measure $\beta$ supported on $S$, whenever $\alpha \sim_W \beta$, we have:
\[
\int c \, d\beta \leq \int c \, d\alpha.
\]
A coupling $\pi \in \Pi(\sigma_X, \sigma_Z)$ is $(c, W)$-monotone if its support is $(c, W)$-monotone.

By Theorem 3.6 of \cite{zaev2015}, an optimal measure for problem \eqref{eqn: fixed barycenter problem} is necessarily $(c, W)$-monotone. Using this property, we now prove that left-monotone couplings are optimal for costs of martingale Spence–Mirrlees type.


\begin{theorem}\label{thm: structure of optimizer in three period limit problem}
	Assume $c(x, z) \in C(X \times Z)$ is differentiable in $x$ and satisfies the martingale Spence–Mirrlees condition. 
	If $\pi(x,z)$ is optimal for the localized problem \eqref{eqn: fixed barycenter problem} with marginals $\sigma_X$ and $\sigma_Y$, 
	then $\pi$ is left-monotone.
\end{theorem}

\begin{proof}
	Since $\pi$ is optimal, it must be $(c, W)$-monotone by Theorem 3.6 of \cite{zaev2015}. 
	Let $\Ga$ be the support of $\pi$ and assume, for contradiction, that $(x, z^-), (x, z^+), (x', z') \in \Ga$ with $z^- < z^+$ and $z' \in (z^-, z^+)$ for some $x < x'$. 
	We can write $z' = (1 - \la) z^- + \la z^+$ for some $\la \in (0, 1)$. 
	
	Define the measure $\bt$ supported on $\Gamma$ and construct a competitor measure $\al$:
	\begin{align*}
		\bt &= (1 - \la) \de_{(x, z^-)} + \la \de_{(x, z^+)} + \de_{(x', z')}\\
		\al &= (1 - \la) \de_{(x', z^-)} + \la \de_{(x', z^+)} + \de_{(x, z')}.
	\end{align*}
	
	Define the function:
	\[
	k(t) \coloneqq (1 - \la) c(t, z^-) + \la c(t, z^+) - c(t, z').
	\]
	Since $c$ satisfies the Spence–Mirrlees condition, $\partial_x c(x, z)$ is strictly concave in $z$, ensuring that $k(t)$ is differentiable with:
	\begin{align*}
		k'(t) &= (1 - \la) \partial_x c(t, z^-) + \la \partial_x c(t, z^+) - \partial_x c(t, z')\\
		&< (1 - \la) \partial_x c(t, z^-) + \la \partial_x c(t, z^+) - \left[(1 - \la) \partial_x c(t, z^-) + \la \partial_x c(t, z^+)\right] = 0.
	\end{align*}
	Hence, $k(t)$ is strictly decreasing. Since $x < x'$, we obtain:
	\begin{align*}
		\int c \, d\al &= (1 - \la)c(x', z^-) + \la c(x', z^+) + c(x, z')\\
		&< (1 - \la)c(x, z^-) + \la c(x, z^+) + c(x', z') = \int c \, d\bt.
	\end{align*}		  

	This contradicts the $(c, W)$-monotonicity of $\Ga$.
	Thus, the assumption that $z' \in (z^-, z^+)$ for $x < x'$ must be false, implying that $\Gamma$ is left-monotone.
\end{proof}

Under reasonable conditions, the preceding result implies uniqueness of the optimal plan, which can in fact be constructed fairly explicitly.  Proofs of similar results for the martingale optimal transport plans can be found in  \cite{bj} and \cite{HenryLabordereTouzi19}; these can be adapted with minimal changes to problem \eqref{eqn: fixed barycenter problem}.

Rather than modify these arguments, we offer here a slightly different proof of uniqueness of a left montone coupling $\pi \in \Pibar{F,\sigma_X,\sigma_Z}$, which, though requiring somewhat stronger assumptions, we feel offers complementary intuition to the proofs in \cite{bj} and \cite{HenryLabordereTouzi19}.

\begin{theorem}\label{thm: left monotone uniqueness}
  Assume that $X$ is an interval and $Z=Z_- \cup Z_+$ is the union of two intervals $Z_- =[\underline{z_-}, \overline{z_-}]$, and $Z_+ =[\underline{z_+}, \overline{z_+}]$ with  $\overline{z_-} < y <\underline{z_+}$, where $y= \int_Z zd\sigma_Z(z)$. Furthermore, assume that $\mu_X$ is non-atomic and $c$ satisfies the martingale Spence-Mirrlees condition. 
  Then there exists a unique solution $\pi \in \Pibar{F, \sigma_X, \sigma_Z}$ to \eqref{eqn: fixed barycenter problem}, where $F$ is the constant function $F(x) =y$.

\end{theorem}
The proof requires the following lemma:

\begin{lemma}\label{lem: left monotone structure}
 Under the assumptions in Theorem \ref{thm: left monotone uniqueness}, let $\pi =\sigma_X\otimes\ka_X^Z \in \Pibar{F, \sigma_X, \sigma_Z}$.  Then the conditional probability $\ka_X^Z$ is supported on two points for $\sigma_X$ a.e. x, $\ka_X^Z(x,dz) =\alpha_- \delta_{T_-(x)} +\alpha_+ \delta_{T_+(x)}$. Furthermore, $T_-:X \rightarrow Z_-$ is a decreasing mapping while $T_+:X \rightarrow Z_+$ is increasing.
\end{lemma}

\begin{proof}
The barycenter condition implies that the conditional probability must be supported on at least two points, one in each of $Z_-$ and $Z_+$.

Suppose some $x$ is coupled to three points, that is, $(x,z_i) \in \tsp(\pi)$ for three points $z_0<z_1<z_2$.  Assume that $z_1 \in Z_-$ (the argument for $z_1 \in Z_+$ is very similar).

The left monotonicity implies that no points $z \in (z_0,z_2)$ can be coupled to $x' >x$.  

Now, the barycenter condition implies that every $\tilde x <x$ must couple to at least two points, and one of these, $\tilde z_+$ must belong to $Z_+$.  Any other $\tilde z_0$ which couples to $\tilde x$ must satisfy $\tilde z_0  >z_1$, as otherwise  $ 
 \tilde z_0  <z_1 <\tilde z_+$, violating left monotonicity. 

 The above considerations imply that \emph{only} $x$ may couple with points in $(z_0,z_1)$.  Since $\sigma_X(\{x\})=0$ by assumption, we must have $\sigma_{Z}((z_0,z_1))=0$.  Now, there are at most countably many intervals within $Z$ satisfying this, so there are at most countably many points $x$ that couple with three or more points.  Thus, almost every $x$ gets coupled to exactly two points, which we may denote by $T_\pm(x) \in Z_\pm$.  The desired monotonicity then follows from the left monotonicity.


\end{proof}
We can now prove Theorem \ref{thm: left monotone uniqueness}.
\begin{proof}
The argument is an adaptation of the standard proof of uniqueness in the classical optimal transport problem, found in, for example, \cite{santambrogio2015}.

Note that Theorem \ref{thm: structure of optimizer in three period limit problem} and Lemma \ref{lem: left monotone structure} imply that any solution is concentrated on the graphs of two functions $T_-$ and $T_+$. 

Now, if there are two optimal couplings, $\pi_0$ and $\pi_1$, both must concentrate on pairs of graphs $T_+^0,T_-^0$ and $T_+^1,T_-^1$, resepctively. 
Linearity implies that $\pi_{1/2} =\frac{1}{2}[\pi_0+\pi_1]$ is also optimal in \eqref{eqn: fixed barycenter problem}.  It must too then concentrate on a pair of graphs $T_+^{1/2},T_-^{1/2}$.  However,  it clearly concentrates on the union of the graphs of  $T_+^0,T_-^0, T_+^1$ and $T_-^1$; this is possible only if $T_-^0=T_-^1:=T_-$ and $T_+^0=T_+^1:=T_+$. 
Now, to finish the proof, we claim there is only one $\pi \in \Pibar{F,\sigma_X,\sigma_Z}$ which is concentrated on these two graphs.  This follows as each conditional probability $\ka_X^Z(x,dz)=\lambda_-\delta_{T_-(x)}+\lambda_+\delta_{T_+(x)}$ of $\pi =\sigma_X \otimes \ka_X^Z $ must satisfy $y=\lambda_-T_-(x)+\lambda_+T_+(x)$, which uniquely determines $\lambda_- =\frac{y-T_+(x)}{T_-(x) -T_+(x)}$ and $\lambda_+ =\frac{y-T_-(x)}{T_+(x) -T_-(x)}$.

\end{proof}

\section{Structural results for three-period MOT}
We now develop uniqueness and structural results for several three-period problems, all under rather specific conditions on the cost and at least some of the marginals.  As in subsection \ref{sect: pointwise coupling structure}, we will assume $d=1$; that is, the marginals are supported on compact subsets $X,Y,Z \subseteq \mathbb{R}$ of the real line.

\subsection{Uniqueness of the optimal coupling for $|\tsp(\mu_Y)| = 2$}

When 
$\mu_Y$ is a discrete measure on $\R$ such that $|\tsp(\mu_Y)| = 2$, 
we establish the following  result.

\begin{lemma}\label{lem:singleton}
	If $\mu_X \co \mu_Y$ and $|\tsp(\mu_Y)| = 2$, the set of martingale couplings $\PiM(\mu_X, \mu_Y)$ is a singleton.
\end{lemma}

\begin{proof}
	Let $\pi = \mu_X \ot \kappa_X^Y \in \PiM(\mu_X, \mu_Y)$.
	and let $y_1, y_2$ be the two points in $\tsp(\mu_Y)$. 
	The disintegration of $\pi$ is given by $\kappa_X^Y(x, dy) = g_1(x) \delta_{y_1}(dy) + g_2(x) \delta_{y_2}(dy)$.
	By the law of total probability and the martingale condition, we obtain for $\mu_X$ a.e. $x$:
	\begin{align*}
		1 &= g_1(x) + g_2(x), \\
		 x &= y_1 g_1(x) + y_2 g_2(x).
	\end{align*}
	This system has a unique solution:
	\[
	g_1(x) = \frac{y_2 - x}{y_2 - y_1}, \quad g_2(x) =  \frac{x - y_1}{y_2 - y_1}.
	\]
	Hence, $\PiM(\mu_X, \mu_Y)$ contains exactly one element.
\end{proof}


\begin{theorem}\label{discMarUniq}
	Let $\mu_X \co \mu_Y \co \mu_Z$, where $\mu_X$ is non-atomic, $\mu_Y = a_1\delta_{y_1} +a_2\delta_{y_2}$ and $Z=[\underline z_-, \overline z_-] \cup [\underline z_+, \overline z_+]$ 
    with $\overline z_- <y_1 <y_2<\underline z_+$. 
	 Suppose $c(x, y, z) = c_1(x, y) + c_2(y, z) + c_3(x, z)$, where $c_1, c_2, c_3$ are  continuous 
	and the partial derivative $(c_3)_x$ exists and satisfies the martingale Spence-Mirrlees condition.
	Then the three-period MOT problem \eqref{prob:Primal-multi-period-MOT} with cost $c$ has a unique optimal solution $\pi$.
\end{theorem}
\begin{proof}
	By Lemma \ref{lem:singleton}, 
	the martingale coupling $\pi^{XY}$ between $\mu_X$ and $\mu_Y$ is unique.  Uniqueness of the optimal $\pi$ will then follow if we can establish uniqueness of an optimal coupling between $\pi^{XY}$ and $\mu_Z$.

It is straight forward to see that, conditioning the optimizer $\pi=\mu_Y\otimes \ka^{XZ}_Y$ on $y$, the conditional coupling $\ka^{XZ}_Y(y,dxdz)$ must be optimal between its marginals $\ka^{X}_Y(y,dx)$ and $\ka^{Z}_Y(y,dz)$in \eqref{eqn: fixed barycenter problem} for cost $c_3(x,z)$ for each of $y_0$ and $y_1$.  Clearly $\ka^{X}_Y(y_0,dx)$ and $\ka^{X}_Y(y_1,dx)$ must both be non-atomic, so that Lemma \ref{lem: left monotone structure} implies $\ka^{XZ}_Y(y_i,dx)$ concentrates on two graphs, $T^i_+:X \rightarrow Z_+$ and $T^i_-:X \rightarrow Z_-$.  Therefore, any optimal $\pi$ concentrates on two graphs $T_+:X\times Y \rightarrow Z_+$ and $T_-:X\times Y \rightarrow Z_-$ over $(x,y)$.  The proof of uniqueness is then essentially identical to the proof of uniqueness in Theorem \ref{thm: structure of optimizer in three period limit problem}.

\end{proof}

\subsection{Uniqueness of the optimal coupling for $|\tsp(\mu_Y)| = 3$}


Throughout this subsection, we will make the following assumptions:

\begin{enumerate}
    \item[A1] $X=[\underline{x},\overline{x}]$ and $\mu_X$ is absolutely continuous with respect to Lebesgue measure, with density $\frac{d\mu_X}{dx}$.
    \item[A2] $\mu_Y =\sum_{i=0}^2a_i\delta_{y_i}$ is supported on the three points $y_0,y_1$ and $y_2$.
    \item[A3] $Z=Z_- \cup Z_+$ is the union of two intervals $Z_- =[\underline{z_-}, \overline{z_-}]$, and $Z_+ =[\underline{z_+}, \overline{z_+}]$ with  $\overline{z_-} < y_0<y_1<y_2 <\underline{z_+}$ and  $\mu_Z$ is non-atomic. 
    \item[A4] The cost function takes the form $c(x,y,z) = f(x,y)z^2$, where $f$ is a differentiable function  with $\partial_xf,\partial_yf<0$.
\end{enumerate}


We will also use the following notation: for a given $\pi \in \PiM(\mu_x,\mu_y,\mu_Z)$,  $\nu =f_\#\pi$ will denote the distribution of $w=f(x,y) \in W:=f(X,Y)$, and $\gamma =\big((x,y,z) \rightarrow (f,z)\big)_\#\pi$ will denote the coupling between $\nu$ and $\mu_Z$ induced by $\pi$.
\begin{lemma}\label{lem: full support optimality}
Assume A1-A4 and suppose $\pi$ is optimal in the three period MOT \eqref{prob:Primal-multi-period-MOT}.  There there is an optimal coupling $\tilde \pi$ such that for all $(x,y,z),(x',y',z),  
    \in spt(\pi)$ and $f(x,y)=f(x',y')$  $(x',y', z) \in spt(\tilde \pi)$.
\end{lemma}

\begin{proof}
The cost only depends on the coupling $\gamma$ between $\nu$ and $\mu_Z$.  It therefore suffices to construct a martingale measure $\tilde \pi \in \PiM(\mu_X,\mu_Y,\mu_Z)$ with $\tilde \pi^{XY} = \pi^{XY}$ (and consequently $\tilde \nu=\nu$) and $\tilde \gamma =\gamma$ with the desired property.

Disintegrate $\pi^{XY} =\nu\otimes\ka^{XY}_W$ and $\gamma =\nu\otimes \ka^Z_W$ with respect to $\nu$; we will work below conditional on $w=f(x,y)$.  We need only to find a martingale coupling $\tilde \ka^{XYZ}_W$ between each conditional probability $\ka^{XY}_W$ and $ \ka^Z_W$, which will ensure that the resulting $\tilde \pi =\nu \otimes \tilde \ka^{XYZ}_W \in \PiM(\mu_X,\mu_Y,\mu_Z)$, such that $\tsp(\tilde \ka^{XYZ}_W(w,dxdydz) = \tsp(\tilde \ka^{XY}_W(w,dxdy)) \times \tsp(\tilde \ka^{z}_W(w,dz))$.

Since $\mu_Y$ is supported on three points, $y_0, y_1,y_2$, we have $\ka^{XY}_W(w,dxdy) =\sum_{i=0}^2 \alpha_i \delta_{(x_i,y_i)}$, where each $f(x_i,y_i)=w$.  The structure of $Z$ implies that each $\ka^Z_W(w,dz) =\beta_+\ka^{Z+}_W(w,dz) +\beta_-\ka^{Z-}_W(w,dz) $ where $\ka^{Z\pm}_W(w,dz) \in \cP(Z_\pm) $ and $\beta_+ +\beta_-=1$.  Clearly, setting $E_\pm = \int_{Z_\pm} z\ka^{Z\pm}_W(w,dz)$, we have each $y_i \in (E_-,E_+)$, and so there exist $\lambda^i_\pm>0$ with $\lambda^i_-+\lambda^i_+=1$ such that $y_i = \lambda^i_-E_-+\lambda^i_+E_+$.  We then build the conditional probability $\tilde \ka^{XYZ}_W =\sum_{i=0}^2 \alpha_i \delta_{(x_i,y_i)} \otimes(\lambda^i_- \ka^{Z_-}_W +\lambda^i_+\ka^{Z_+}_W) $.  By construction, this yields a martingale coupling.  We only need to check that it has the correct $Z$ marginal, $\ka_W^Z$.  We do this by summing over the three values of $i$ to get 
\begin{eqnarray*}
\sum_{i=0}^2\alpha_i(\lambda^i_- \ka^{Z-}_W  +\lambda^i_+\ka^{Z+}_W)&=& (\sum_{i=0}^2\alpha_i\lambda^i_- )\ka^{Z-}_W  +(\sum_{i=0}^2\alpha_i\lambda^i_+)\ka^{Z+}_W
\end{eqnarray*}
So, we need to show $\sum_{i=0}^2\alpha_i\lambda^i_\pm = \beta_\pm .$ The martingale condition for the original conditional coupling $\ka^{XYZ}_W$ yields
\begin{eqnarray*}
    \beta_+E_+ +\beta_-E_- &=&\sum_{i=0}^2\alpha_iy_i\\
    &=&\sum_{i=0}^2\alpha_i(\lambda^i_-E_-+\lambda^i_+E_+)=\sum_{i=0}^2\alpha_i(\lambda^i_-)E_-+\sum_{i=0}^2\alpha_i(\lambda^i_+E_+)
\end{eqnarray*}
which, by uniqueness of decompositions into convex combinations of two point in $\mathbb{R}$, yields the desired result.
\end{proof}








\begin{lemma}\label{lem: graphical stucture of reduced coupling}
Assume A1-A4 and suppose $\pi$ is optimal in the three period MOT \eqref{prob:Primal-multi-period-MOT}.  Then $\gamma = (f,z)_\#\pi$ 
is concentrated on a graph over $Z$.
\end{lemma}

\begin{proof}
   Using Lemma \ref{lem: full support optimality}, we get that the support of $\tilde \pi$ satisfies the $(c,W)$ optimality condition; very similar arguments to Theorem \ref{thm: structure of optimizer in three period limit problem} and Theorem \ref{thm: left monotone uniqueness} yield that $\gamma$ is left monotone and consequently concentrated on the union of two graphs, $T_+:W \rightarrow Z_+$ increasing and $T_-:W \rightarrow Z_-$.  Consequently, it concentrates on the graph of $G:Z \rightarrow W$, defined by $G_{Z_\pm} =T_\pm^{-1}$. 
\end{proof}

\begin{theorem}\label{thm: uniqueness for triply supported y}
Under assumptions assumption A1-A4, the solution to the three period MOT problem is unique.
\end{theorem}
\begin{proof}
The first part of the proof is a fairly standard application of the graphical structure.  Suppose that $\pi_0$ and $\pi_1$ are both solutions.  Then, by linearity, so is $\pi_{1/2} =\frac{1}{2}[\pi_0+\pi_1]$.  By Lemma 
\ref{lem: graphical stucture of reduced coupling}, the corresponding distributions $\nu_0$ and $\nu_1$ of $w=f(x,y)$ must be coupled to $\mu_Z$ by graphs $G_0$ and $G_1$ respectively; that is, $\gamma_i(w,z) = (Id,G_i)_\#\mu_Z$.  Similarly, since $\pi_{1/2}$ is optimal, the distributions $\nu_{1/2} =\frac{1}{2}[\nu_0 +\nu_1]$ must be coupled to $\mu_Z$ by a graph $G_{1/2}$, $\gamma_{1/2}(w,z) = (Id,G_{1/2})_\#\mu_Z$. Since the coupling $\gamma_{1/2} = \frac{1}{2}[\gamma_0+\gamma_1]$ concentrates on the \emph{union} of the graphs of $G_0$ and $G_1$, this is only possible if $G_0=G_1 =G_{1/2}$.  We then must have 
$$
\nu_0 =(G_0)_\#\mu_Z = (G_1)_\#\mu_Z = \nu_1
$$
It remains to show that the distribution $\nu$ uniquely determines the coupling $\pi$.  In fact, since the coupling between $\nu$ and $\mu_Z$ is uniquely determined by the argument above, 
we must only show: 
\begin{enumerate}
    \item That the $(x,y)$ marginal $\pi^{XY}$ is uniquely determined by $\nu$; that is, that for a given $\nu$, there is a unique $\pi^{XY} \in \PiM(\mu_X,\mu_Y)$ such that $\nu = f_\#\pi^{XY}$, and;
    \item That the coupling $\pi$ is uniquely determined by $\pi^{XY}$ and the coupling $\gamma$ between $\nu = f_\#\pi^{XY}$ and $\mu_Z$.
\end{enumerate} 

Part 2 above follows fairly easily from the structure of $\gamma$.  Indeed, it is enough to show uniqueness of the condition probabilities $\ka_W^{XYZ}(w, dxdydz)$ of $\pi=\nu\otimes\ka_W^{XYZ}$ coupling the conditional probabilities $\ka_W^{XY}(w, dxdy)$ of $\pi^{XY}=\nu\otimes\ka_W^{XY}$ and $\ka_W^{Z}(w, dxdy)$ of $\mu_Z{}=\nu\otimes\ka_W^{Z}$ for $\nu$ a.e $w$.  Disintegrating with respect to $(x,y)$, $\ka_W^{XYZ} =\ka_W^{XY}\otimes \ka_{WXY}^{Z}$, the proof of Lemma \ref{lem: graphical stucture of reduced coupling} implies that each $\ka_{WXY}^{Z}(x,y,w=f(x,y), dz) =\alpha_-\delta_{T_-(f(x,y))}+\alpha_+\delta_{T_+(f(x,y))}$ is supported on the two points $T_\pm (f(x,y))$.  Now, the martingale constraint requires $\int z \ka_{WXY}^{Z}(x,yw=f(x,y), dz) =y$, which uniquely determines the weights $\alpha_\pm$

Part 1 is more involved; we turn to this task now.

Since $\mu_X$ is absolutely continuous and $\mu_Y$ is supported on three points $\{y_1, y_2, y_3\}$, 
the disintegration of $\pi^{XY} =\mu_X \otimes \ka_X^Y$ can be written as:
\[
\ka^Y_X = q_0(x) \delta_{(x, y_0)} + q_1(x) \delta_{(x, y_1)} + q_2(x) \delta_{(x, y_2)},
\]
where $q_0(x), q_1(x), q_2(x)$ are non-negative weights satisfying the constraints:
\be\label{eqn: three point conditional probability}
q_0(x) + q_1(x) + q_2(x) = 1, \quad
q_0(x)y_0 + q_1(x)y_1 + q_2(x)y_2 = x .
\ee

Now note that $\nu$ has support contained in $W=[\underline{x} +y_0, \overline{x} +y_2]$.  For all $w$ there are at most three points $(x_i,y_i)$ such that $f(x_i,y_i) =w$, and for $w >f(\underline{x},y_1)$, there is only one such point, $(x_0,y_0)$.  For these $w$, the change of variables equation between $\mu_X$ and $\nu$ reads 
$$
\frac{d\nu}{dw} = \frac{q_0(x_0)\frac{d\mu_X}{dx}(x_0)}{\frac{\partial f}{\partial x}(x_0,y_0)} 
$$
which then uniquely determines $q_0(x_0) =\frac{\frac{d\nu}{dw}(w)\frac{\partial f}{\partial x}(x_0,y_0) }{\frac{d\mu_X}{dx}(x_0)}$.  Inserting this into \eqref{eqn: three point conditional probability} then determines $q_1(x_0)=\frac{y_2-x-q_0(x_0)(y_2-y_0)}{y_2-y_1}$ and $q_2(x_0)=\frac{y_1-x-q_0(x_0)(y_1-y_0)}{y_1-y_2}$ as well, for all $x_0$ such $f(x_0,y_0) <f(\underline{x},y_1)$.  

For other $w$, there are at most three points $(x_i,y_i)$ such that $f(x_i,y_i) =w$, in which case our change of variables formula reads given by $\frac{d\nu}{dw} = \sum_{i=0}^2\frac{q_i(x_i)\frac{d\mu_X}{dx}(x_i)}{\frac{\partial f}{\partial x}(x_i,y_i)} $.  This equation may be solved for $q_0(x_0)$ in terms of $q_1(x_1)$ and $q_2(x_2)$. Noting that $x_1,x_2 < x_0$, we may therefore boot strap to determine $q_0(x)$ for larger values of $x$ using the solutions for smaller ones; a precise argument is as follows.

Suppose by way of contradiction that there exists some $x$ such that the $q_i$ are not uniquely determined at $x$.  We let $x_u$ be the infimum of the set of such
$x$.  Note that $f(x_u,y_0)\leq f(\underline x, y_1)$.  We can choose $x_0\geq x_u$ close enough to $x_u$ such that the $q_i$ are not uniquely determined at $x_0$, but for $f(x_i,y_i) =f(x_0,y_0)$ for $i=1$, and possibly $2$ as well, we  have $x_i <x_u$, so the $q_i$ are uniquely determined at $x_1$ and $x_2$. The change of varibles equation $\frac{d\nu}{dw} = \sum_{i=0}^2q_i(x_i)\frac{\partial f}{\partial x}(x_i,y_i) \frac{d\mu_X}{dx}(x_i)$ then uniquely determines $q_0(x_0)$, and \eqref{eqn: three point conditional probability} then determines $q_1(x_0)$ and $q_2(x_0)$, contradicting the assumption and completing the proof.

\end{proof}

\subsection{Uniqueness of the optimal coupling for discrete $\mu_X$ and $\mu_Y$}

In this section, we consider the case when $\mu_X \co \mu_Y \co \mu_Z$, 
where $\mu_X$ and $\mu_Y$ are discrete probability measures on $\mathbb{R}$, 
and $\mu_Z$ is an absolutely continuous probability measure on $\mathbb{R}$. 
Let $\mu_X$ and $\mu_Y$ be supported on countable sets $\{x_i\}$ and $\{y_j\}$, respectively. 
We consider a bounded continuous cost function $c(x,y,z)$ satisfying that for any fixed pairs $(x_i, y_j) \neq (x_k, y_\ell)$, 
the function $c(x_i, y_j, z) - c(x_k, y_\ell, z)$ intersects any given line at most countably many times.

By Theorem 5.2 of \cite{NutzStebeggTan20}, 
there is no duality gap between the primal and dual formulations of the multi-period MOT problem, 
and an optimal dual solution exists. 
For each $\pi \in \PiM(\mu_X, \mu_Y, \mu_Z)$, define:
\[
	I_{ij}^{\pi} \coloneqq \{z \in \tsp(\mu_Z) \mid (x_i, y_j, z) \in \tsp(\pi) \}.
\] 
Fix a dual optimizer $(u, v, w, g, h)$ for the dual problem \eqref{prob:Dual-multi-period-MOT}, where $u, v, g, h$ are functions on $\{x_i\}$ and $\{y_j\}$, while $w$ is a function on $\mathbb{R}$. Define:
\[
	I_{ij}^c \coloneqq \{z \in \tsp(\mu_Z) \mid u_i + v_j + w(z) + g_i(y_j - x_i) + h_{ij}(z - y_j) - c(x_i, y_j, z) = 0\}.
\]

\begin{lemma}\label{zAtomlessUnqSources}
 If $\pi$ is an optimal solution to the martingale optimal transport problem \eqref{prob:Primal-multi-period-MOT} with cost $c$, then for any distinct pairs $(i,j) \neq (k, \ell)$, we have $\mu_Z(I_{ij}^{\pi} \cap I_{k\ell}^{\pi}) = 0$.
\end{lemma}

\begin{remark}
	In particular, if $x_i + y_j \neq 0$ for all $i,j$ and $c(x, y, z) = (x + y)z^2$, the lemma applies.
\end{remark}

\begin{proof}
	If $(x_i, y_j, z)$ and $(x_k, y_{\ell}, z)$ both belong to $\tsp(\pi)$, then $z$ belongs to the intersection $I_{ij}^{\pi} \cap I_{k\ell}^{\pi}$. 
	From the dual optimality conditions, we have:
	\begin{align}
		u_i + v_j + w(z) + g_i(y_j - x_i) + h_{ij}(z - y_j) &= c(x_i, y_j, z), \label{ijzDual}\\
		u_k + v_{\ell} + w(z) + g_k(y_{\ell} - x_k) + h_{k\ell}(z - y_{\ell}) &= c(x_k,y_{\ell}, z). \label{kellzDual}
	\end{align}
	Subtracting \eqref{ijzDual} from \eqref{kellzDual} and grouping terms independent of $z$, we obtain:
	\[
		D_{ij}^{k\ell} + (h_{k\ell} - h_{ij})z + c(x_i, y_j, z) - c(x_k,y_{\ell}, z) = 0.
	\]
	By assumption on the cost function, the term $c(x_i, y_j, z) - c(x_k, y_{\ell}, z)$ has only countably many intersections with any linear function in $z$. 
	Since $I_{ij}^{\pi} \cap I_{k\ell}^{\pi} \subset I_{ij}^c \cap I_{k\ell}^c$, 
	and $\mu_Z$ is absolutely continuous, we conclude that $\mu_Z(I_{ij}^{\pi} \cap I_{k\ell}^{\pi}) = 0$.
\end{proof}

With this, we can establish uniqueness of the optimal coupling.

\begin{theorem}\label{thm: uniqueness for discrete x and y}
	The optimal measure $\pi$ for the MOT problem \eqref{prob:Primal-multi-period-MOT} with cost $c$ is unique. 
\end{theorem}

\begin{proof}
	By Lemma \ref{zAtomlessUnqSources}, each $z$ is associated with a unique pair $(i,j)$ for $\mu_Z$-almost all $z$. 
	Thus, any optimal solution $\pi$ is concentrated on the graph of a function; that is, 
	there exist functions $x: Z \to X$ and $y: Z \to Y$ such that $\tsp(\pi) \subset \{(x(z), y(z), z) \mid z \in \tsp(\mu_Z)\}$.

	Now, assume for contradiction that there exist two optimal measures $\pi^0, \pi^1 \in \PiM(\mu_X,\mu_Y,\mu_Z)$. 
	Since the objective function is linear in $\pi$, their convex combination $\pi^t = t \pi^0 + (1 - t) \pi^1$	is also an optimal solution fpr all $t \in (0,1)$. 
	However, since each optimal solution must be concentrated on a function graph, $\pi^t$ would be supported on two distinct graphs $\tsp(\pi^0) \cup \tsp(\pi^1)$.
	By Lemma \ref{zAtomlessUnqSources}, for $\mu_Z$-almost all $z$, 
	there is a unique pair $(x(z), y(z))$ such that $(x(z), y(z), z) \in \tsp(\pi)$. 
	This implies that the two graphs must coincide, meaning $\pi^0 = \pi^1$.
	Thus, the optimal coupling $\pi$ is unique.
\end{proof}

\section{First-order approximation of price bounds using real market data}

In this section, we build on Proposition~\ref{prop: derivatives of optimal cost}, Theorem~\ref{thm: optimizers of limiting three period problem}, and Remark~\ref{rem: linearization} to explore applications in the context of financial modeling.

We are interested in providing price bounds for a three-period, path-dependent derivative with payoff
\[
c(x,y,z) = c_1(x,y) + c_2(y,z) + c_3(x,z),
\]
under given marginals $\mu_X$, $\mu_Y$, and $\mu_Z$.

Computing exact pricing bounds for such cost functions is generally challenging, as little is known about the solution for such costs. 
Direct numerical approximation of the three-period MOT problem using the Sinkhorn or ODE method (as in \cite{hiew2024ordinary}) may be computationally expensive due to high-dimensional input data and the curse of dimensionality.

As in Section 3, we introduce a parameter $\ep$ to $c_3$, and define the function:
\[
P_l(\ep) = \inf_{\pi \in \Pi_M(\mu_X, \mu_Y, \mu_Z)} \int (c_1 + c_2 + \ep c_3) \, d\pi.
\]
As mentioned in Remark~\ref{rem: linearization}, the first-order approximation to $P_l(\ep)$ around 0 is given by:
\begin{align*}
	P_l(\ep) &\approx Q_l(\ep) \coloneqq P_l(0) + \ep P_l'(0)\\
	&= \int c_1 \, d\pi^{XY} + \int c_2 \, d\pi^{YZ} + \ep \int c_3 \, d\pi,
\end{align*}
where $\pi^{XY}$ and $\pi^{YZ}$ are the optimal two-period martingale couplings for the MOT problem with costs $c_1$ and $c_2$, respectively. The measure $\pi$ is a joint distribution consistent with $\pi^{XY}$ and $\pi^{YZ}$, and its conditional probability $\kappa^{XZ}_Y(y,dxdz)$ given $y$ solves the variant of the MOT problem \ref{eqn: fixed barycenter problem} for the cost $c_3$ with the conditional probabilities of $\pi^{XY}$ and $\pi^{YZ}$ given $y$ being the marginals (recall Proposition \ref{prop: derivatives of optimal cost} and Theorem \ref{thm: optimizers of limiting three period problem}).

We denote the corresponding upper bound for the true cost and its first-order approximation by $P_u(\ep) = \sup_{\pi \in \Pi_M(\mu_X, \mu_Y, \mu_Z)} \int (c_1 + c_2 + \ep c_3) \, d\pi.$ and $Q_u(\ep)$, respectively.

The main advantage of this approximation is that each of the three summands in $Q_l(\ep)$ and $Q_u(\ep)$ is computed using a two-period MOT problem or the variant \eqref{eqn: fixed barycenter problem} of a two-period MOT problem. 
For some of these problems, the optimizer is well understood; for martingale Spence-Mirrlees type costs, for example, the optimizer is given by the left-monotone coupling as a result of \cite{bj}, or our Theorem~\ref{thm: structure of optimizer in three period limit problem}. Even for more general costs, solving three two-period MOT problems is significantly more efficient than solving the full three-period MOT problem.\footnote{If we discretize each of the measures $\mu_X,\mu_Y$ and $\mu_z$ with $N$ points, the number of unknowns in the three period MOT problem used to compute $P_l$ is $N^3$, while it is $N^2$ in each of the two period MOT problems needed to determine $Q_l$, so the total number in the three problems needed is $3N^2$.}

We illustrate this with two numerical examples involving time-dependent derivatives written on Amazon stock. We extract option-implied risk-neutral distributions $\mu_X, \mu_Y$ and $\mu_Z$ from option prices observed on November 23\textsuperscript{rd}, 2022, using maturities of December 16\textsuperscript{th}, 2022, January 20\textsuperscript{th}, 2023, and February 17\textsuperscript{th}, 2023, following the method of \cite{BreedenLitzenberger}.

Real world prices are often estimated by constructing particular martingales $\pi_m \in \Pi^M(\mu_X, \mu_Y, \mu_Z)$, using particular modeling assumptions. The corresponding price $P_m(\epsilon) :=\int (c_1 + c_2 + \ep c_3) \, d\pi_m$ then clearly depends on the particular assumptions used to construct $\pi_m$, but we must always have $P_l(\epsilon) \leq P_m(\epsilon) \leq P_u(\epsilon)$.  We compute prices using one such modelling method here, and compare the resulting price with our first order approximations $Q_l$ and $Q_u$ of  $P_l$ and $P_u$, respectively.  Specifically, we use a tree-like construction. Trees are a common way to construct martingales in financial modeling, in which each timestep is constructed independently by allowing each value to jump either up or down by an equal increment with equal probability (see, for example, \cite{hull2016options}). Our precise construction here is slightly different, as we must preserve the single time marginals. Our idea is construct each transition probability between the marginals with
 small local movements of the underlying asset, as consecutive times in our data are quite close together. To ensure the martingale and marginal constraints are respected, we formulate a linear program to find a martingale coupling minimizing the deviation $c(x, y) = |y - x|^p$ (for $p=1,2,3$) over each time step.\footnote{In fact, this model amounts to finding the single timestep couplings $\pi^{XY}$ and $\pi^{YZ}$ by solving the two period MOT problem \eqref{prob:Primal-multi-period-MOT} between $\mu_X$ and $\mu_Y$ and $\mu_Y$ and $\mu_Z$, respectively, with costs $|x-y|^p$ and $|y-z|^p$, respectively.  The model martingale, $\pi_m$ in the notation of this section, is then constructed as the Markovian glueing described in Remark \ref{rem: markovian glueing}.} 
 In the case of $p = 2$, any martingale measure will optimize the cost and the linear programming solver will return an extreme point of the set of martingale couplings.
\subsection{Example: Third moment of the sum}

We consider the following cost function:
\[
\bar c(x, y, z) = (x + y + z)^3 = x^3 + y^3 + z^3 + 3(x^2y + x^2z + y^2x + y^2z + z^2x + z^2y) +6xyz
\]

Note that the payoff of a three period Asian option depends on the risk neutral distribution of the sum $x+y+z$; pricing these options therefore depends on the properties of this distribution. Its first and second moments are uniquely determined by the marginals and martingale condition, respectively, and so the expected value of $\bar c$ represents the first non-fixed moment, and therefore has an important impact on the pricing of Asian options.  The third moment is closely related to the option-implied skewness. In financial markets, skewness is used to quantify tail risk and is frequently referenced in practice as a measure of asymmetry in the return distribution \cite{CBOE2011}. It has also been used as a proxy for physical skewness in forecasting expected returns \cite{ConardDittmarGhysels13}, \cite{JackwerthRubinstein96}.  Our method, detailed below, allows us to approximate model independent bounds on the third moment in closed form.

Under a risk-neutral pricing framework where the pricing kernel is always a martingale measure, the expectation of the cost function $\bar c$ under any martingale measure $\pi$ with marginals $\mu_X$, $\mu_Y$, and $\mu_Z$ simplifies to:
\begin{align*}
\int \bar c(x, y, z) \, d\pi &= 7\int x^3 \, d\mu_X + 4\int y^3 \, d\mu_Y + \int z^3 \, d\mu_Z\\
&\qquad \qquad + 9\int xy^2 \, d\pi^{XY} + 3\int yz^2 \, d\pi^{YZ} + 3\int xz^2 \, d\pi^{XZ},
\end{align*}
where $\pi^{XY}$, $\pi^{YZ}$, and $\pi^{XZ}$ are the bivariate marginals of $\pi$.  As the first three terms above are completely determined by the marginals and are therefore equal for all $\pi \in \PiM(\mu_X,\mu_Y,\mu_Z)$, we neglect them for simplicity and consider only the cross terms $xy^2, yz^2$ and $xz^2$.
We therefore define the following cost function:
\[
c(x, y, z, \ep) = c_1(x, y) + c_2(y, z) + \ep c_3(x, z),
\]
where
\[
c_1(x, y) = 9xy^2, \quad
c_2(y, z) =  3yz^2, \quad
c_3(x, z) = 3xz^2.
\]
We recover the original problem $\bar c$  by setting $\ep = 1$.

Since $c_1(x, y)$ and $c_2(y, z)$ are both Spence–Mirrlees type costs, the corresponding two-period optimizers $\pi^{XY}$ and $\pi^{YZ}$ are left-monotone martingale couplings, which can be constructed explicitly \cite{bauerle2019martingale}, \cite{bj}. Theorem \ref{thm: structure of optimizer in three period limit problem} implies that the optimal conditional probabilities $\kappa^{XZ}_Y(y, dxdz)$ of $\pi^{XZ}$ given $y$ are also left monotone and can be constructed in closed form as well. 
Now, when $\ep = 0$, the cost reduces to $c_1 + c_2$, and by Proposition \ref{prop:ConsecutiveOptimality}, the twofold marginals of any optimal plan must agree with the specified marginals $\pi^{XY}$ and $\pi^{YZ}$. Thus, the values of $P_l(0) = Q_l(0)$ and $P_u(0) = Q_u(0)$ are computable explicitly, as are $P_l'(0)$ and $P_u'(0)$, using Proposition \ref{prop: derivatives of optimal cost}.   At $\ep = 1$, we then use the first-order approximation $Q_l(1)$ and $Q_u(1)$ to estimate the price bounds $P_l(1)$ and $P_u(1)$ of the full cost function $\bar c$.

Table \ref{tab:asian_option} summarizes the price bounds  (for the cross terms) computed via the first-order approximation ($Q_l$, $Q_u$) and the prices returned by the tree-like model using deviation costs $|y - x|^p$ for $p = 1, 2, 3$. We observe that all values computed using the tree-like method lie within the interval $Q_l(1)$ and $Q_u(1)$, demonstrating that the first-order approximation provides reasonable bounds.  
 As this is a relatively small scale problem, we can also solve the 3 period MOT problem numerically by linear programming. The true lower and upper bounds  are $P_l(1)= 14,314,844$ and $P_u(1)=14,323,889$,  which are extremely close to our approximate bounds $Q_l(1)=14,314,867$ and $Q_u(1) =14,323,889$, respectively.   

This information is represented graphically in Figures \ref{fig:skew_bounds_zoom_0} and \ref{fig:skew_bounds_zoom_1}. We note that, as is fairly common in derivative pricing, the prices coming from all models are fairly close together (since $P_l(\epsilon)$ is quite close to $P_u(\epsilon)$, and all model curves $P_m(\epsilon)$ must lie between them), so it would be difficult to distinguish different curves for the full range of $\epsilon$ visually; we therefore present only zoomed in views of the graphs for values of $\epsilon$ near $0$ and $1$.





\begin{table}[h!]
	\centering
	\begin{tabular}{|c|c|c|c|c|c|}
		\hline
		$\varepsilon$ & $Q_l(\varepsilon)$ & $p = 1$ & $p = 2$ & $p = 3$ & $Q_u(\varepsilon)$ \\
		\hline
		0 & 11,448,994 & 11,454,263 & 11,450,363 & 11,451,641 & 11,455,414 \\
		\hline
		1 & 14,314,867 & 14,322,158 & 14,316,863 & 14,318,572 & 14,323,889 \\
		\hline
	\end{tabular}
	\caption{First-order approximation vs. tree-like method for the third moment of the sum (sum of cross terms only.}
	\label{tab:asian_option}
\end{table}


\begin{figure}[h!]
    \centering
    \includegraphics[width=\textwidth]{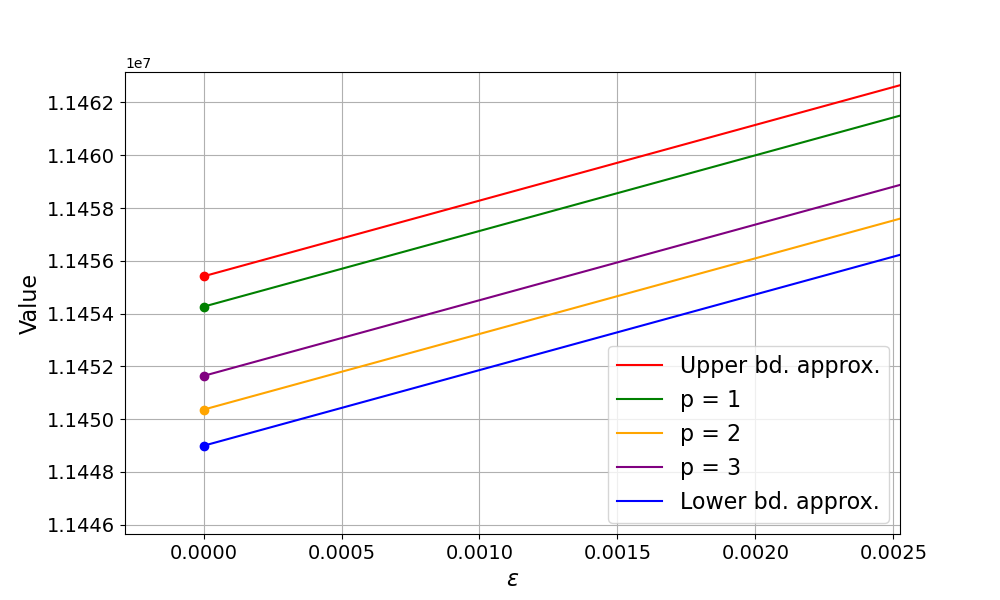}
    \caption{Zoomed view around $\ep=0$ for first-order approximation vs. tree-like method for the third moment of a sum (sum of cross term only).}
    \label{fig:skew_bounds_zoom_0}
\end{figure}

\begin{figure}[h!]
    \centering
    \includegraphics[width=\textwidth]{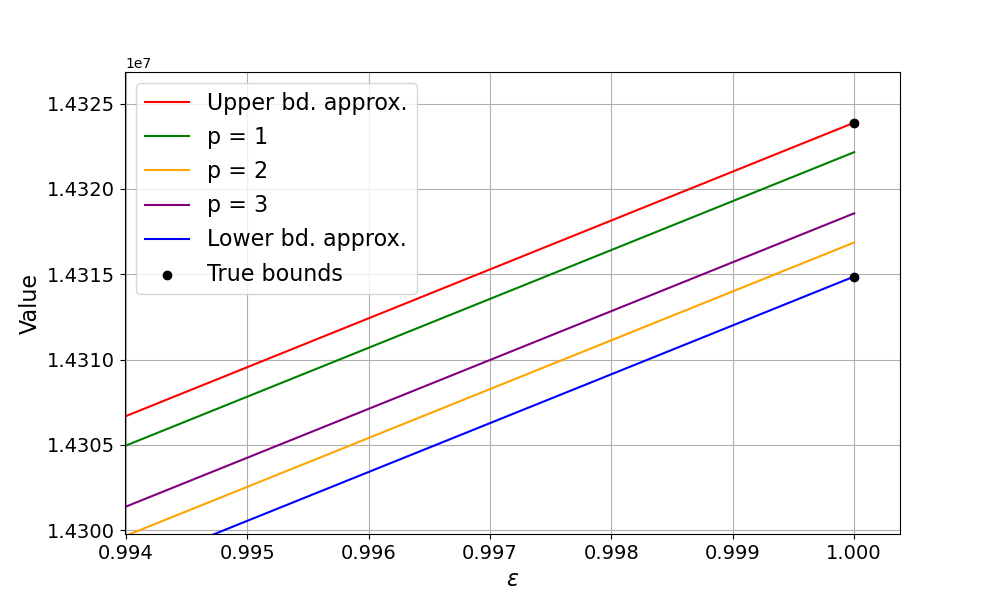}
    \caption{Zoomed view around $\ep=1$ for first-order approximation vs. tree-like method for the third moment of a sum (sum of cross term only).}
    \label{fig:skew_bounds_zoom_1}
\end{figure}

\subsection{Example: Basket of straddle options}

We consider a basket of forward start straddle options, combining payoffs over all pairwise periods $(x,y)$, $(y,z)$, and $(x,z)$. The cost function is given by $c(x,y,z) = c_1(x,y) + c_2(y,z) + c_3(x,z)$ where 
\[
c_1(x,y) = |y - x|, \quad c_2(y,z) = |z - y|, \quad c_3(x,z) = |z - x|.
\]
Straddle options are widely used in financial markets \cite{de2011exotic}, and have been studied in the context of MOT \cite{hobson2012robust}, \cite{hobson2015robust}.

Unlike the example in the preceding subsection, the cost functions $c_1, c_2$ and $c_3$ here are not of martingale Spence–Mirrlees type, and we cannot rely on left-monotonicity results to construct the optimal coupling. Instead, we compute the couplings $\pi^{XY}$, $\pi^{YZ}$, and $\pi^{XZ}$ in Theorem \ref{thm: optimizers of limiting three period problem}, needed to compute $Q_l$ and $Q_u$, using a linear programming solver; this amounts to numerically solving 3 two-dimensional linear programs, which is much more tractable than the three-dimensional linear program required to find the exact values $P_l$ and $P_u$. 

Table \ref{tab:straddle_option} summarizes the first-order approximations and the prices obtained via the tree-like method. We observe that, for $\ep = 1$, all three values from the tree-like method with $p = 1, 2, 3$ lie within the first-order approximation bounds $Q_l(1)$ and $Q_u(1)$, supporting the idea that the first-order expansion provides a good approximation for the price bound.

Figure \ref{fig:straddle_bounds} shows that the computed prices using the tree-like method lie strictly between the lower and upper bounds across the entire range $\ep \in [0,1]$.  In fact, the $p=1$ curve stays very close to the approximate lower bound curve, while the $p=3$ stays close to the approximate upper bound curve; it is difficult to distinguish them visually.  The fact that $Q_l(0) =P_m(0)$ for the $p=1$ model is entirely expected (as both are computed by solving the same MOT problems).  The fact that these curves remain close for other values of $\epsilon$, as well as the fact that the $p=3$ model curve is close to the approximate upper bound $Q_u(\epsilon)$ is more surprising.

 For this small scale problem, we can also solve the 3 period MOT problem directly by linear programming, although it is less efficient than computing our approximate bounds. The true lower and upper bounds are $P_l(1)=14.8651$ and $P_u(1)=20.6914$, which are quite close to our approximate values, $Q_l(1)=15.01453$ and $Q_u(1)=20.2759$, respectively.

\begin{table}[h!]
	\centering
	\begin{tabular}{|c|c|c|c|c|c|}
		\hline
		$\ep$ & $Q_l(\ep)$ & $p = 1$ & $p = 2$ & $p = 3$ & $Q_u(\ep)$ \\
		\hline
		0 & 8.5036 & 8.5036 & 9.3460 & 12.2669 & 12.2676 \\
		\hline
		1 & 15.01453 & 15.0957 & 16.5204 & 20.1843 & 20.2759 \\
		\hline
	\end{tabular}
	\caption{Comparison of first-order approximation and tree-like method for the straddle cost.}
	\label{tab:straddle_option}
\end{table}

\begin{figure}[h!]
    \centering
    \includegraphics[width=\textwidth]{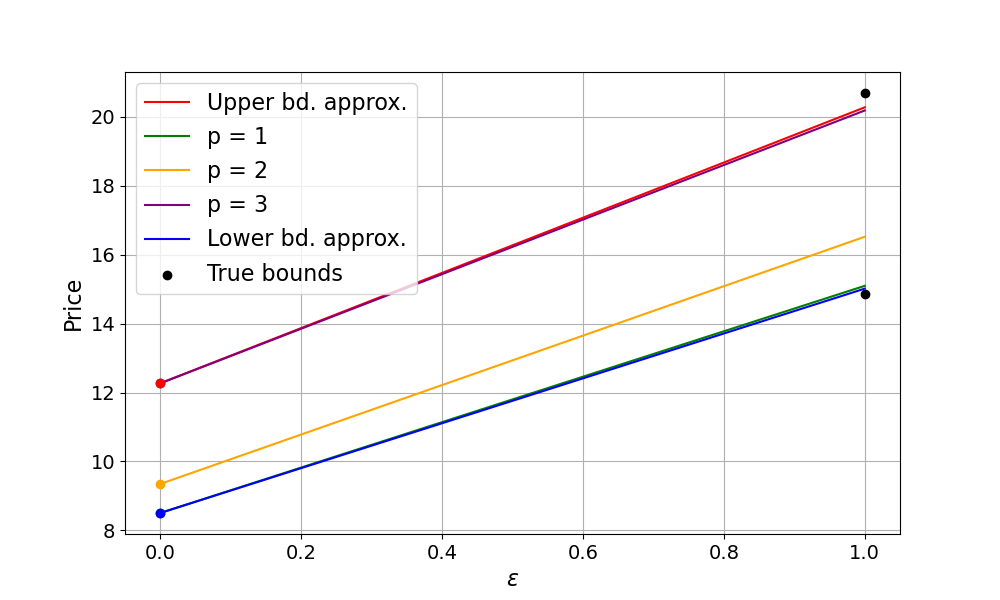}
    \caption{Comparison of first-order approximation and tree-like method for the straddle cost over $\ep \in [0,1]$.}
    \label{fig:straddle_bounds}
\end{figure}

\textbf{First author biography:}\\
 Brendan Pass is a faculty member in the Department of Mathematical and Statistical Sciences at the University of Alberta (Edmonton, Alberta, Canada). He works primarily on optimal transport, and is in particular an expert on multi-marginal problems.   Pass is one of the founders of the Kantorovich Initiative, a nascent organization focused on interdisciplinary optimal transport research and supported by the Pacific Institute for the Mathematical Sciences (PIMS) and the US National Science Foundation (NSF). Pass' 2011 PhD thesis at the University of Toronto garnered him the 2012 Cecil Graham Doctoral Dissertation Award from the Canadian Applied and Industrial Mathematics Society (CAIMS), and he was the 2021 recipient of the  CAIMS - PIMS Early Career Award, in recognition of his contributions to optimal transport theory.

\bibliographystyle{plain}
\bibliography{biblio}

\begin{thebibliography}{10}

\bibitem{bauerle2019martingale}
Nicole B{\"a}uerle and Daniel Schmithals.
\newblock Martingale optimal transport in the discrete case via simple linear
  programming techniques.
\newblock {\em Mathematical Methods of Operations Research}, 90:453--476, 2019.

\bibitem{BeiglbockHenryPenkner2013}
Mathias Beiglb{\"o}ck, Pierre Henry-Labordere, and Friedrich Penkner.
\newblock Model-independent bounds for option prices—a mass transport
  approach.
\newblock {\em Finance and Stochastics}, 17:477--501, 2013.

\bibitem{bj}
Mathias Beiglb{\"o}ck and Nicolas Juillet.
\newblock {On a problem of optimal transport under marginal martingale
  constraints}.
\newblock {\em The Annals of Probability}, 44(1):42 -- 106, 2016.

\bibitem{BreedenLitzenberger}
Douglas~T Breeden and Robert~H Litzenberger.
\newblock Prices of state-contingent claims implicit in option prices.
\newblock {\em Journal of Business}, pages 621--651, 1978.

\bibitem{BruckerhoffJuillet22}
Martin Br{\"u}ckerhoff and Nicolas Juillet.
\newblock Instability of martingale optimal transport in dimension d≥ 2.
\newblock {\em Electronic Communications in Probability}, 27:1--10, 2022.

\bibitem{CBOE2011}
{CBOE}.
\newblock The cboe skew index.
\newblock White paper, 2011.

\bibitem{ConardDittmarGhysels13}
Jennifer Conrad, Robert~F. Dittmar, and Eric Ghysels.
\newblock Ex ante skewness and expected stock returns.
\newblock {\em The Journal of Finance}, 68(1):85--124, 2013.

\bibitem{de2011exotic}
Frans De~Weert.
\newblock {\em Exotic options trading}.
\newblock John Wiley \& Sons, 2011.

\bibitem{Galichon2016}
Alfred Galichon.
\newblock {\em Optimal transport methods in economics}.
\newblock Princeton University Press, Princeton, NJ, 2016.

\bibitem{HenryLabordere}
Pierre Henry-Labord{\`e}re.
\newblock {\em Model-free hedging: A martingale optimal transport viewpoint}.
\newblock CRC Press, 2017.

\bibitem{HenryLabordereTouzi19}
Pierre Henry-Labord\`ere and Nizar Touzi.
\newblock An explicit martingale version of the one-dimensional {B}renier
  theorem.
\newblock {\em Finance Stoch.}, 20(3):635--668, 2016.

\bibitem{hiew2024ordinary}
Joshua Zoen-Git Hiew, Luca Nenna, and Brendan Pass.
\newblock An ordinary differential equation for entropic optimal transport and
  its linearly constrained variants.
\newblock {\em arXiv preprint arXiv:2403.20238}, 2024.

\bibitem{hobson2015robust}
David Hobson and Martin Klimmek.
\newblock Robust price bounds for the forward starting straddle.
\newblock {\em Finance and Stochastics}, 19(1):189--214, 2015.

\bibitem{hobson2012robust}
David Hobson and Anthony Neuberger.
\newblock Robust bounds for forward start options.
\newblock {\em Mathematical Finance: An International Journal of Mathematics,
  Statistics and Financial Economics}, 22(1):31--56, 2012.

\bibitem{hull2016options}
John~C Hull and Sankarshan Basu.
\newblock {\em Options, futures, and other derivatives}.
\newblock Pearson Education India, 2016.

\bibitem{JackwerthRubinstein96}
Jens~Carsten Jackwerth and Mark Rubinstein.
\newblock Recovering probability distributions from option prices.
\newblock {\em The Journal of Finance}, 51(5):1611--1631, 1996.

\bibitem{LeskelaVihola17}
Leskel{\"a} Lasse and Matti Vihola.
\newblock {Conditional convex orders and measurable martingale couplings}.
\newblock {\em Bernoulli}, 23(4A):2784 -- 2807, 2017.

\bibitem{NutzStebeggTan20}
Marcel Nutz, Florian Stebegg, and Xiaowei Tan.
\newblock Multiperiod martingale transport.
\newblock {\em Stochastic Processes and their Applications}, 130(3):1568--1615,
  2020.

\bibitem{santambrogio2015}
Filippo Santambrogio.
\newblock Optimal transport for applied mathematicians.
\newblock {\em Birk{\"a}user, NY}, 55(58-63):94, 2015.

\bibitem{strassen65}
Volker Strassen.
\newblock The existence of probability measures with given marginals.
\newblock {\em The Annals of Mathematical Statistics}, 36(2):423--439, 1965.

\bibitem{TalponenViitasaari2014}
Jarno Talponen and Lauri Viitasaari.
\newblock Note on multidimensional breeden--litzenberger representation for
  state price densities.
\newblock {\em Mathematics and Financial Economics}, 8:153--157, 2014.

\bibitem{Villani09}
C{\'e}dric Villani et~al.
\newblock {\em Optimal transport: old and new}, volume 338.
\newblock Springer, 2009.

\bibitem{zaev2015}
Danila~A Zaev.
\newblock On the {Monge--Kantorovich} problem with additional linear
  constraints.
\newblock {\em Mathematical Notes}, 98:725--741, 2015.

\end{thebibliography}

\end{document}